\documentclass[10pt]{amsart}

\usepackage[T1]{fontenc}
\usepackage{lmodern}
\usepackage{microtype}
\usepackage{xcolor}
\usepackage{amsmath,amssymb,mathtools}
\usepackage{enumitem}
\usepackage[hidelinks]{hyperref}

\allowdisplaybreaks
\setlist[enumerate,1]{leftmargin=2.4em,label=(\roman*),itemsep=0.35em,topsep=0.35em}
\setlist[itemize]{leftmargin=2em,itemsep=0.25em,topsep=0.35em}

\makeatletter
\def\ps@headings{\ps@empty
  \def\@evenhead{%
    \setTrue{runhead}%
    \normalfont\scriptsize\rlap{\thepage}\hfil}%
  \def\@oddhead{%
    \setTrue{runhead}%
    \normalfont\scriptsize\hfil\llap{\thepage}}%
  \let\@mkboth\markboth
}
\makeatother
\newtheorem{theorem}{Theorem}[section]
\newtheorem{proposition}[theorem]{Proposition}
\newtheorem{lemma}[theorem]{Lemma}
\newtheorem{corollary}[theorem]{Corollary}
\theoremstyle{definition}
\newtheorem{definition}[theorem]{Definition}

\theoremstyle{remark}
\newtheorem{remark}[theorem]{Remark}
\newtheorem{warning}[theorem]{Warning}

\newcommand{\GL}{\operatorname{GL}}
\newcommand{\Vect}{\operatorname{Vect}}
\newcommand{\Fun}{\operatorname{Fun}}
\newcommand{\Aut}{\operatorname{Aut}}
\newcommand{\Nat}{\operatorname{Nat}}

\newcommand{\Pic}{\operatorname{Pic}}

\newcommand{\Tr}{\operatorname{Tr}}
\newcommand{\id}{\mathrm{id}}
\newcommand{\disc}{\mathrm{disc}}
\newcommand{\uSm}{u\text{-}\mathrm{sm}}

\newcommand{\KM}{K^{\mathrm M}}

\newcommand{\Fq}{\mathbb F_q}
\newcommand{\calI}{\mathcal I}
\newcommand{\calV}{\mathcal V}

\newcommand{\calL}{\mathcal L}
\newcommand{\calO}{\mathcal O}
\newcommand{\calC}{\mathcal C}

\newcommand{\op}{\mathrm{op}}
\newcommand{\fd}{\mathrm{fd}}

\newcommand{\ex}{\mathrm{ex}}

\newcommand{\PreHilb}{\operatorname{PreHilb}}

\begin{document}
\title{The  principal series 2-representation for $\GL_n\times\GL_n$ over a 2-dimensional local field}
	
\author{Xuecai Ma}
\address{
Westlake Institute for Advanced Study, Hangzhou, Zhejiang 310024, P.R. China \\
\hspace*{1em} Institute for Theoretical Sciences, Westlake University, Hangzhou, Zhejiang 310030, P.R. China
}
\email{maxuecai@westlake.edu.cn}

\author{Jingwen Zhu}
\address{
 Westlake University, Hangzhou, Zhejiang 310030, P.R. China
}

\email{zhujingwen@westlake.edu.cn}


\begin{abstract}
Let $F=\Fq((u))((t))$ and
$H=\GL_n(F)_L\times\GL_n(F)_R$.  An ordered tuple $\alpha=(\alpha_1,\dots, \alpha_n) $ defines a
cross-$K_2$ multiplier on the doubled Borel.  Twisted equivariantization then
yields an exact uniformly smooth categorical principal series
$2$-representation of $H$.   For unitary parameters,  the $\mathbb R((X))$-measure 
 gives a positive Hom-valued formal Hermitian inner product on the
spherical part.
These constructions provide two-dimensional categorical analogues of selected
unramified structures.
\end{abstract}

\maketitle
\tableofcontents

\section{Introduction}

Let $K=k((u))$ be a one-dimensional local field with finite residue field.  In the unramified local Langlands correspondence for
$\GL_n$, an unramified Frobenius-semisimple Weil--Deligne parameter
\[
(\phi,N),\qquad
\phi:W_{K}\longrightarrow \GL_n(\mathbb C),
\qquad \phi|_{I_{K}}=1,\quad N=0,
\]
corresponds to an irreducible unramified smooth representation $\pi_\phi$ of
$\GL_n(K)$.  With a fixed Frobenius and normalized Satake convention, both sides
are encoded by the semisimple conjugacy class
$s_\phi=\phi(\operatorname{Frob})$ in the dual group
$\widehat {\GL_n}=\GL_n(\mathbb C)$, equivalently by its unordered eigenvalues.
The spherical line $\pi_\phi^{\GL_n(\mathcal O_K)}$ is one-dimensional, the commutative
spherical Hecke algebra $\mathcal H(\GL_n(K), \GL_n(\mathcal O_K))$ acts on it through the Satake
character determined by $s_\phi$, and the standard unramified factor is
$\det(1-Ts_\phi)^{-1}$.  Thus the classical picture ties together a Weil
parameter, a smooth representation, a spherical fixed line, a Hecke
eigencharacter, and a scalar local factor
\cite{bump1998automorphic,getz2024introduction}.

Several complementary programs motivate a two-dimensional extension of this
picture.  Kapranov proposed that the automorphic objects in
such a higher-dimensional theory should be categorified, replacing ordinary
representation spaces by linear categories and group representations by higher
representations \cite{kapranov1995analogies}.  Parshin's higher local class field theory and higher-adelic methods
provide an abelian $K$-theoretic and adelic framework for fields and schemes of
dimension two; his later formulation of a two-dimensional Langlands program
emphasizes its relation to the classical correspondence through direct images
\cite{parshin2012questions}.  In the unramified local setting,
Osipov made part of this philosophy concrete by constructing a categorical
analogue of unramified principal series for groups over a two-dimensional local
field, using central extensions and associated noncommutative reciprocity laws
\cite{osipov2013unramified}. The present paper asks for a categorical analogue of selected parts of the unramified local Langlands  picture
over the two-dimensional local field $F$.  Following Kapranov's idea, our representation takes values   in 2-vector spaces. Our  corresponding group is $\GL_n(F)  \times  \GL_n(F)$,   whereas Kapranov's uses $\GL_{2n}(F)$.   Another important distinction is that Osipov defines
his congruence filtration using the rank-one discrete valuation ring
of the two-dimensional local field $F$, whereas the present   
construction uses the rank-two valuation ring.

Let $\Lambda=\mathbb C((X))$, with $X$ a formal variable. We set
$G=\GL_n(F)$, $H=G_L\times G_R$, and
$B_H=B_L\times B_R$. Our use of   $\GL_n \times \GL_n$ is motivated by Fesenko's zeta integral for two-dimensional local fields. One can see \cite{fesenko2003analysis,fesenko2008adelic,fesenko2010analysis} for more details.   The left and right factors supply the two diagonal
inputs used in the cross-$K_2$ symbol.  For an ordered tuple
$\alpha=(\alpha_1,\ldots,\alpha_n)\in(\Lambda^\times)^n$, the two-step degree
$\nu_F=v_u\circ\partial_t$ defines the normalized multiplier
\[
\sigma_\alpha(b,b')
=
\prod_{r=1}^n
\alpha_r^{-\nu_F\{x_r(b),y_r(b')\}},
\qquad b,b'\in B_H.
\]
For a rank-two congruence level $C$, put $S_C=H/C$, and write
$\calO_2=k[[u]]+t\,k((u))[[t]]$ and
$H_0=\GL_n(\calO_2)_L\times\GL_n(\calO_2)_R$. Using $\sigma_{\alpha }$,  we construct a  category $\mathcal I^{\uSm}_F(\alpha )$,   which admits an action of $H$.  This can be thought  of as the two-dimensional analogue of principal series representations.
  Accordingly, the categorical action is constructed
by twisted equivariantization and filtered pseudocolimits, while
higher-local-field integration enters separately on its spherical part.

The roles of the principal objects may be compared as follows.  The table isS
a guide to the construction, not a correspondence between its two columns.
\begin{center}
\small
\renewcommand{\arraystretch}{1.18}
\begin{tabular}[htbp]{@{}p{0.20\textwidth}p{0.30\textwidth}p{0.38\textwidth}@{}}
\hline
Object & Classical one-dimensional theory & Present two-dimensional model\\
\hline
Local field
& $K=k((u))$
& $F=k((u))((t))$ with two ordered valuation directions\\
Acting group
& $\GL_n(K)$
& $H=\GL_n(F)_L\times\GL_n(F)_R$\\
Parameter
& Semisimple class $s_\phi$, or unordered Frobenius eigenvalues
& Ordered tuple $\alpha$ entering the cross-$K_2$ multiplier $\sigma_\alpha$;
  no Weil parameter is constructed\\
Representation
& Smooth representation $\pi_\phi$ on a vector space
& Abelian category $\calI_F^{\uSm}(\alpha)$ with an exact uniformly smooth
  $2$-representation of $H_{\disc}$\\
Spherical object
& The fixed line $\pi_\phi^{\GL_n(\mathcal O_K)}$
& A spherical object $v_\alpha$ with coherent right $H_0$-equivariance\\
Algebraic model
& Spherical Hecke algebra $\mathcal H(
\GL_n(K),  \GL_n(\mathcal O_K))$ and the Satake isomorphism
& Twisted $B_H$-equivariantization of
  $\Fun(S_C,\Vect_\Lambda)$ at level $C$; no Satake isomorphism is asserted\\

Inner product
& Haar measure for spherical convolution and, in the unitary case, a scalar
  invariant Hermitian pairing
& Fesenko--Morrow integration and a Hom-valued formal Hermitian pairing on the
  spherical part\\
\hline
\end{tabular}
\end{center}

The multiplication lines $L_\alpha(b)$ supply the composition constraints of a weak
$(B_H)_{\disc}$-action on each split category
$\Fun(H/C,\Vect_\Lambda)$.  Twisted $B_H$-equivariantization gives the
fixed-level categorical principal series.  We equip $H$ with the Hausdorff
subgroup topology generated by finite intersections of conjugates of
rank-two congruence subgroups and form an explicit filtered pseudocolimit.
The resulting category $\calI_F^{\uSm}(\alpha)$ is abelian, and strict right
translation defines an exact uniformly smooth $2$-representation of
$H_{\disc}$.  The construction also produces a spherical object.  After a
model of elementary $2$-Tate objects and a compatible graded determinant
extension have been fixed, precomposition with the determinant $2$-group
projection gives a formal inflation whose Picard kernel acts coherently
trivially.

For $h$ normalizing a fixed level $C$, right translation is an exact
endofunctor $R_h^C$ of $\calI_C(\alpha)$.  We use the intrinsic fixed-level
categorical trace
\[
\Tr_{\mathrm{cat}}(R_h^C)
=\Nat\bigl(\id_{\calI_C(\alpha)},R_h^C\bigr).
\]

For the spherical part, using the coefficientwise involution on $\Lambda$ and
unitary parameters, we use Fesenko's
integration on $F$ and Morrow's bi-invariant integral on $\GL_n(F)$ to
construct a positive-definite Hom-valued formal inner product.  We compute
the general-rank congruence and Borel-congruence volumes, prove isometry under
level refinement, and obtain a unitary action of the integral subgroup
$H_0$.  For $n=2$, the volume formula specializes to Waller's normalization.

Sections~2 and~3 construct the cross-$K_2$ multiplier and its weak categorical
action.  Sections~4 and~5 construct the uniformly smooth pseudocolimit, its
right action, the formal determinant inflation, and the spherical object.
The final section records the fixed-level categorical trace and establishes
the general-rank Fesenko--Morrow spherical inner product and its specialization
to $n=2$.
The construction is motivated by higher local fields, gerbal
representations, and proposed two-dimensional unramified Langlands theories
\cite{fesenko2000invitation,frenkel2012gerbal,osipov2013unramified}, but it
does not prove a local Langlands correspondence, a Satake isomorphism, a
Hecke-eigenobject theorem, or a nontrivial determinant-kernel action.  The
categorical, Tate-theoretic, and integration-theoretic ingredients are
drawn from
\cite{baez2004higher,braunling2016tate,braunling2018index,
ganter2008representation,ganter2015inner,
fesenko2003analysis,morrow2008integration,morrow2010integration,
waller2019measure,kato1979generalization}.

\textbf{Acknowledgements:}
The authors thank Ivan Fesenko for his support and discussions on higher local fields. This work was supported by Westlake University grants 207010146022301 and 207010145609901, and by the Hangzhou Postdoctoral Daily Funding program.

\section{The two-step Milnor degree and the doubled Borel multiplier}

\subsection{The field and its rank-two valuation}
We  let $k=\Fq, E=k((u)), F=E((t)) $, and let  the rank-one and rank-two rings  be
$\calO_1=E[[t]],  \calO_2=k[[u]]+tE[[t]]$. 
Let $v_t:F^\times\to\mathbb Z$ and $v_u:E^\times\to\mathbb Z$ be normalized valuations.  We equip $\mathbb Z^2$ with the $t$-first lexicographic order: for $(a,b)\in\mathbb Z^2$,
\[
(a,b)\ge (0,0)
\quad\Longleftrightarrow\quad
b>0\ \text{or}\ (b=0\ \text{and}\ a\ge0).
\]
For $x\in F^\times$, write
\[
v(x)=\left(v_u\!\left(\overline{x\,t^{-v_t(x)}}\right),\,v_t(x)\right)\in\mathbb Z^2,
\]
where the bar denotes reduction modulo $t$ of the $t$-adic unit $x\,t^{-v_t(x)}$. With the order just specified,
\[
\calO_2=\{0\}\cup\{x\in F^\times:v(x)\ge(0,0)\}.
\]

For $a,b \in F^{\times}$, the tame boundary is
\[
\partial_t\{a,b\}
=
(-1)^{v_t(a)v_t(b)}  \overline{a^{v_t(b)}b^{-v_t(a)}}
\in E^\times.
\]
Define
\[
\nu_F=v_u\circ\partial_t:\KM_2(F)\longrightarrow\mathbb Z.
\]
Our orientation is $\nu_F\{u,t\}=1$.

\begin{lemma}[Determinant formula]\label{lem:detformula}
For $a,b,c,d\in\mathbb Z$,
\[
\nu_F\{u^at^b,u^ct^d\}=ad-bc.
\]
The pairing
\[
F^\times\times F^\times\longrightarrow\mathbb Z,\qquad
(x,y)\longmapsto\nu_F\{x,y\}
\]
is biadditive and alternating.
\end{lemma}

\begin{proof}
The tame boundary is
\[
\partial_t\{u^at^b,u^ct^d\}=(-1)^{bd}u^{ad-bc}.
\]
The sign has $u$-valuation zero, so applying $v_u$ gives the formula. Additivity follows from bilinearity of the Milnor symbol and the boundary map. For every $x\in F^\times$,
\[
\partial_t\{x,x\}=(-1)^{v_t(x)^2},
\]
whose $u$-valuation is zero. Hence $\nu_F\{x,x\}=0$, and the bilinear pairing $\nu_F\{\cdot,\cdot\}$ is alternating.
\end{proof}

Thus
\[
\nu_F=v_u\circ\partial_t:\KM_2(F)\longrightarrow\mathbb Z
\]
is a surjective homomorphism, since $\nu_F\{u,t\}=1$. For every $\beta\in\Lambda^\times$, the formula
\[
\chi_\beta(\xi)=\beta^{\nu_F(\xi)}
\]
defines a group character $\KM_2(F)\to\Lambda^\times$. No topology on $\KM_2(F)$ is used below.

\subsection{The doubled Borel cocycle}
Let
\[
G=\GL_n(F),\qquad H=G_L\times G_R.
\]
Let $B\subset G$ be the upper triangular Borel and set $B_H=B_L\times B_R$. For $b=(b_L,b_R)\in B_H$, define
\[
x_r(b)=(b_L)_{rr},\qquad y_r(b)=(b_R)_{rr}.
\]
Fix an ordered tuple
\[
\alpha=(\alpha_1,\ldots,\alpha_n)\in(\Lambda^\times)^n.
\]

\begin{definition}
The cross-$K_2$ multiplier attached to $\alpha$ is the following function on the underlying abstract group of $B_H$:
\[
\sigma_\alpha(b,b')=
\prod_{r=1}^n\alpha_r^{-\nu_F\{x_r(b),y_r(b')\}},
\qquad b,b'\in B_H.
\]
\end{definition}

\begin{proposition}
The function $\sigma_\alpha$ is a normalized scalar 2-cocycle on $B_H$:
\[
\sigma_\alpha(bb',b'')\sigma_\alpha(b,b')
=
\sigma_\alpha(b,b'b'')\sigma_\alpha(b',b''),
\]
and $\sigma_\alpha(1,b)=\sigma_\alpha(b,1)=1$.
\end{proposition}

\begin{proof}
Diagonal coordinates are multiplicative on a Borel subgroup. For each $r$, bilinearity gives
\[
\nu_F\{x_r(bb'),y_r(b'')\}
=
\nu_F\{x_r(b),y_r(b'')\}
+
\nu_F\{x_r(b'),y_r(b'')\}
\]
and
\[
\nu_F\{x_r(b),y_r(b'b'')\}
=
\nu_F\{x_r(b),y_r(b')\}
+
\nu_F\{x_r(b),y_r(b'')\}.
\]
The cocycle identity follows after exponentiating and multiplying over $r$. Normalization follows from $\{1,z\}=\{z,1\}=0$.
\end{proof}

For each $b\in B_H$, define the standard line
\[
L_\alpha(b)=\Lambda e_b.
\]
The multiplication isomorphisms
\[
m_{b,b'}:L_\alpha(b)\otimes L_\alpha(b')\xrightarrow{\sim}L_\alpha(bb'),
\qquad
m_{b,b'}(e_b\otimes e_{b'})
=\sigma_\alpha(b,b')e_{bb'},
\]
form a multiplicative line bundle, or equivalently the line form of the ordinary central extension defined by $\sigma_\alpha$.  The maps $m_{b,b'}$, rather than the scalars alone, will be the composition constraints of the weak categorical action; $\sigma_\alpha(b,b')$ is their coordinate in the displayed standard bases.

For commuting $b,c\in B_H$, define the commutator
\[
\kappa_\alpha(b,c)=\frac{\sigma_\alpha(c,b)}{\sigma_\alpha(b,c)}.
\]
This convention is the inverse of the other common convention $\sigma(b,c)/\sigma(c,b)$. It depends only on the cohomology class of the multiplier.

For tuples
\[
\lambda_L=((a_1,b_1),\ldots,(a_n,b_n)),\qquad
\lambda_R=((c_1,d_1),\ldots,(c_n,d_n))
\]
in $(\mathbb Z^2)^n$, put
\[
\ell_{\lambda_L}
=
\left(\operatorname{diag}(u^{a_1}t^{b_1},\ldots,u^{a_n}t^{b_n}),I_n\right),
\]
\[
r_{\lambda_R}
=
\left(I_n,\operatorname{diag}(u^{c_1}t^{d_1},\ldots,u^{c_n}t^{d_n})\right).
\]
These elements commute.

\begin{proposition}[Paired commutator formula]\label{prop:paired}
One has
\[
\kappa_\alpha(\ell_{\lambda_L},r_{\lambda_R})
=
\prod_{r=1}^n\alpha_r^{a_rd_r-b_rc_r}.
\]
\end{proposition}

\begin{proof}
The term $\sigma_\alpha(r_{\lambda_R},\ell_{\lambda_L})$ is $1$ because the left diagonal coordinates of the first argument are all $1$. By Lemma~\ref{lem:detformula},
\[
\sigma_\alpha(\ell_{\lambda_L},r_{\lambda_R})
=
\prod_{r=1}^n\alpha_r^{-(a_rd_r-b_rc_r)}.
\]
Taking the ratio proves the formula.
\end{proof}

\begin{proposition}[Gauge invariance]
Let $q:B_H\to\Lambda^\times$ be a normalized 1-cochain and let
\[
\sigma_\alpha^q(b,b')=q(b)q(b')q(bb')^{-1}\sigma_\alpha(b,b').
\]
For commuting $b,c$, the commutator obtained from $\sigma_\alpha^q$ equals $\kappa_\alpha(b,c)$.
\end{proposition}

\begin{proof}
The coboundary factors in the two orders are identical because $bc=cb$ and therefore cancel in the quotient.
\end{proof}

\section{Coherent 2-groups, weak actions, and equivariantization}

\subsection{Coherent 2-groups and 2-representations}
In this subsection we briefly review coherent 2-groups and categorical actions. See~\cite{baez2004higher,baez2012infinite,frenkel2012gerbal} for further background. A coherent 2-group is a monoidal groupoid in which every object is tensor-invertible, with coherent choices of weak inverses. For a $\Lambda$-linear category $\mathcal A$, let
\[
\Aut_\Lambda(\mathcal A)
\]
be the monoidal groupoid whose objects are $\Lambda$-linear autoequivalences, whose morphisms are natural isomorphisms, and whose monoidal product is composition. After choosing an adjoint-equivalence datum for each autoequivalence, this is a coherent 2-group. Moreover, different choices give equivalent coherent 2-groups. If $\mathcal A$ is abelian, $\Aut^{\ex}_\Lambda(\mathcal A)$ denotes the full 2-subgroup on exact autoequivalences.

\begin{theorem}\cite[Corollary 44]{baez2004higher}
	 There is a one-to-one correspondence between equivalence classes of coherent 2-groups and isomorphism classes of quadruples $(G,A,\varphi,[a])$ consisting of
	 \begin{enumerate}
	 	 \item a group $G$,
	 	 \item an abelian group $A$,
	 	 \item an action $\varphi:G\to\Aut(A)$ by automorphisms,
	 	 \item a class $[a]\in H^3_{\mathrm{grp}}(G,A)$ for the preceding $G$-module structure on $A$.
	 \end{enumerate}
\end{theorem}

\begin{definition}\label{def:2rep}
Let $\mathcal  A$ be a $\Lambda$-linear category. A $\Lambda$-linear 2-representation of a coherent 2-group $\mathcal G$ on $\mathcal A$ is a strong monoidal functor
\[
\rho:\mathcal G\longrightarrow\Aut_\Lambda(\mathcal A).
\]
An ordinary group $Q$ is regarded as the discrete strict 2-group $Q_{\disc}$, having objects $q\in Q$ and only identity morphisms. A 2-representation of $Q_{\disc}$ is therefore a weak categorical action of $Q$: autoequivalences $T_q$,  composition constraints $T_qT_{q'}\Rightarrow T_{qq'}$, and the usual pentagon and unit coherences.
\end{definition}

\begin{definition}
	Suppose that $Q$ is a topological group with a chosen basis $\mathcal C_{Q}$ of open subgroups.
	\begin{enumerate}
		\item For a strict action $\rho$ of $Q$ on an abelian category $\mathcal A$, we say that $\rho$ is \emph{uniformly smooth} if every object $X$ is literally fixed by some $C_X\in\calC_Q$ and every morphism $u$ is literally fixed by some $C_u\in\calC_Q$; that is, $\rho(c)X=X$ for $c\in C_X$ and $\rho(c)u=u$ for $c\in C_u$.
		
		 \item More generally, suppose $p:\widetilde{\mathcal G}\to Q_{\disc}$ is a monoidal projection and a 2-representation factors as $\widetilde\rho=\rho\circ p$. We call $\widetilde\rho$ \emph{$p$-uniformly smooth} if the underlying $Q$-action $\rho$ is uniformly smooth.
	\end{enumerate}

\end{definition}

A normalized scalar 2-cocycle on an ordinary group and the associator of a nontrivial coherent 2-group occupy different cohomological degrees. For a specified $Q$-module structure on an abelian group $A$, a class in $H^2_{\mathrm{grp}}(Q,A)$ classifies group extensions inducing that action; such an extension is central when the action is trivial. In the scalar case below, $Q$ acts trivially on $\Lambda^\times$, and the cocycle supplies the tensorator of a weak action of $Q_{\disc}$. In a skeletal coherent 2-group with $\pi_0=Q$, $\pi_1=A$, and fixed $Q$-action on $A$, the associator is represented by a normalized 3-cocycle and its equivalence class lies in $H^3_{\mathrm{grp}}(Q,A)$~\cite{baez2004higher}. If the kernel is a graded Picard groupoid, the full extension also contains its nontrivial $\pi_0$ or degree data. We will use the $H^2$-class $[\sigma_\alpha]$ for Borel equivariantization and discuss the determinant 2-group separately; no identification between these layers is asserted.

\subsection{A scalar multiplier as a weak categorical action}
Let a group $Q$ act on a set $S$ on the left. Let $\sigma\in Z^2_{\mathrm{grp}}(Q,\Lambda^\times)$ be normalized, and let $\{L(q)\}_{q \in Q}$ be one-dimensional $\Lambda$-vector spaces with multiplication isomorphisms
\[
m_{q,q'}:L(q)\otimes L(q')\xrightarrow{\sim}L(qq')
\]
whose scalar in chosen trivializations is $\sigma(q,q')$. Choose standard vectors $e_q\in L(q)$ satisfying
\[
m_{q,q'}(e_q\otimes e_{q'})=\sigma(q,q')e_{qq'}.
\]
Let
\[
\calV(S)=\Fun(S,\Vect_\Lambda).
\]
For $q\in Q$,  we define an exact cocontinuous autoequivalence
\[
T_q:\calV(S)\longrightarrow\calV(S),\qquad (T_qV)_s=L(q)\otimes V_{q^{-1}s}.
\]
There is a composition constraints
\[
\mu_{q,q'}:T_qT_{q'}\Longrightarrow T_{qq'}
\]
whose component at $s$ is $m_{q,q'}\otimes\id$. The unit isomorphism is
\[
u_V:T_1V=L(1)\otimes V\xrightarrow{\sim}V,\qquad
e_1\otimes v\longmapsto v.
\]

\begin{proposition}[Multiplier action]\label{prop:multiplier-action}
The data $(T_q,\mu_{q,q'},u)$ define an exact weak 2-representation
\[
T^\sigma:Q_{\disc}\longrightarrow\Aut^{\ex}_\Lambda\bigl(\calV(S)\bigr).
\]
The pentagon for the composition constraints is exactly the scalar 2-cocycle identity for $\sigma$.
\end{proposition}

\begin{proof}
At $s\in S$ one has
\[
(T_qT_{q'}V)_s=L(q)\otimes L(q')\otimes V_{(qq')^{-1}s},
\]
so $m_{q,q'}$ gives the required coherence isomorphism. On three factors, the two composites from $L(q)\otimes L(q')\otimes L(q'')$ to $L(qq'q'')$ have scalars
\[
\sigma(q,q')\sigma(qq',q'')
\quad\text{and}\quad
\sigma(q',q'')\sigma(q,q'q''),
\]
which agree precisely by the cocycle identity. Normalization makes the displayed maps $u_V$ compatible with the composition isomorphisms  and gives the unit coherences. Exactness and cocontinuity follow from pullback along a bijection and tensoring by a line.
\end{proof}

\subsection{Equivariantization and the twisted action groupoid}

\begin{definition}
We  define  the equivariantization  category $\calV(S)^{Q,\sigma}$  whose objects are pairs $(V,\varepsilon)$ with $V\in\calV(S)$ and natural isomorphisms
\[
\varepsilon_q:T_qV\xrightarrow{\sim}V,\qquad q\in Q
\]
such that
\[
\varepsilon_{qq'}\circ\mu_{q,q'}(V)
=
\varepsilon_q\circ T_q(\varepsilon_{q'}),
\qquad
\varepsilon_1=u_V.
\]
Morphisms are maps in $\calV(S)$ commuting with every $\varepsilon_q$.
\end{definition}

Let $Q\ltimes S$ be the action groupoid. A $\sigma$-twisted representation of this groupoid is a family $V_s$ with isomorphisms
\[
\theta_q(s):L(q)\otimes V_s\xrightarrow{\sim}V_{qs}
\]
satisfying
\[
\theta_{qq'}(s)\circ(m_{q,q'}\otimes\id)
=
\theta_q(q's)\circ(\id\otimes\theta_{q'}(s)),
\qquad
\theta_1(s)(e_1\otimes v)=v.
\]

\begin{proposition}[Equivariantization equals twisted descent]\label{prop:equiv-descent}
There is a canonical equivalence
\[
\calV(S)^{Q,\sigma}\simeq\Fun_\sigma(Q\ltimes S,\Vect_\Lambda).
\]
Under this equivalence,
\[
\theta_q(s)=\varepsilon_q(qs),
\qquad
\varepsilon_q(s)=\theta_q(q^{-1}s).
\]
\end{proposition}

\begin{proof}
The displayed formulas are inverse to one another because
\[
(T_qV)_{qs}=L(q)\otimes V_s.
\]
After substituting $s\mapsto qq's$ in the equivariantization identity, one obtains exactly the twisted descent equation. The same formulas identify morphisms.
\end{proof}

Thus the term ``2-representation'' in the principal series construction below records genuine coherence data: the multiplication maps of the multiplicative lines provide the composition constraints of a weak categorical action, their coordinates are the values of the Borel multiplier, and the principal series category is the corresponding equivariantization.

\section{Congruence levels and the uniformly smooth pseudocolimit}

\subsection{A conjugation-stable congruence basis}
Set
\[
H_0=\GL_n(\calO_2)_L\times\GL_n(\calO_2)_R.
\]
For $i,j\ge1$,  we define
\[
K_{ij}=I_n+u^it^jM_n(\calO_2),
\qquad
C^{\mathrm{dbl}}_{ij}=K_{ij,L}\times K_{ij,R}.
\]
\begin{proposition}\label{prop:congruence-infinite}
	For every $i,j\ge1$, $K_{ij}$ is a normal subgroup of $\GL_n(\calO_2)$ and the quotient $\GL_n(\calO_2)/K_{ij}$ is infinite.
\end{proposition}

\begin{proof}
	The set $J_{ij}=u^it^jM_n(\calO_2)$ is a two-sided ideal. The ring $\calO_2=k[[u]]+tE[[t]]$, with the subspace topology induced by the $t$-adic topology on $E[[t]]$, is complete. If $A\in J_{ij}$, then
	\[
	(I_n+A)^{-1}=\sum_{m\ge0}(-A)^m
	\]
	converges $t$-adically. Every term with $m\ge1$ lies in $J_{ij}$ and has $t$-adic order tending to infinity. Moreover,
	\[
	u^it^j\calO_2=u^it^jk[[u]]+t^{j+1}E[[t]],
	\]
	so $J_{ij}$ is closed for this induced topology. Hence the tail lies in $J_{ij}$ and the inverse belongs to $I_n+J_{ij}\subset\GL_n(\calO_2)$. Closure under multiplication and normality follow from the ideal property.
	
	To prove infinitude, already in rank one the elements $1+at^j$, $a\in E$, give distinct classes modulo $1+u^it^j\calO_2$ whenever their coefficients are distinct modulo $u^ik[[u]]$. The quotient $E/u^ik[[u]]$ is infinite.  The diagonal embedding gives the statement for every $n$.
\end{proof}
Let $\calC_H$ be the family of finite intersections of $H$-conjugates of the groups $C^{\mathrm{dbl}}_{ij}$. It is ordered by reverse inclusion. It is closed under finite intersections and conjugation. We use it as a basis of identity neighborhoods; an open subgroup is, by definition, a subgroup containing a member of $\calC_H$.

\begin{lemma}[The congruence subgroup topology]\label{lem:congruence-topology}
The family $\calC_H$ is a conjugation-invariant basis of identity neighborhoods for a Hausdorff group topology $\tau_{\calC}$ on $H$. Equivalently, a subset $U\subseteq H$ is open if and only if, for every $g\in U$, there is a subgroup $C\in\calC_H$ such that $gC\subseteq U$.
\end{lemma}

\begin{proof}
Every member of $\calC_H$ is a subgroup, the family is closed under finite intersections, and $hCh^{-1}\in\calC_H$ for $h\in H$ and $C\in\calC_H$. These properties are precisely the subgroup-basis axioms for multiplication, inversion, and conjugation to be continuous. The resulting topology is Hausdorff because
\[
\bigcap_{C\in\calC_H}C
\subseteq
\bigcap_{i,j\ge1}C^{\mathrm{dbl}}_{ij}
=\{1\};
\]
indeed, a matrix belonging to every $I_n+u^it^jM_n(\calO_2)$ differs from $I_n$ by a matrix whose entries lie in $\bigcap_{i,j\ge1}u^it^j\calO_2=0$.
\end{proof}

\begin{warning}
	Throughout the paper, the words ``open'' and ``uniformly smooth'' refer to $\tau_{\calC}$. No comparison with any of the standard higher topologies on $H$ is asserted.
\end{warning}

\begin{remark}
Morrow constructs a left- and right-invariant lifted integral on $\GL_n(F)$ for every $n$, and Waller gives an explicit $\mathbb R((X))$-valued measure on $\GL_2(F)$ with the same normalization~\cite{morrow2008integration,waller2019measure}.  
\end{remark}
\subsection{Fixed-level split categories}
For $C\in\calC_H$, put
\[
S_C=H/C,
\qquad
\calV_C=\Fun(S_C,\Vect_\Lambda).
\]
We fix Grothendieck universes $\mathbb U\in\mathbb V$. All sets, groupoids, and linear categories called small are $\mathbb U$-small. The field $\Lambda$ and all vector spaces under consideration are $\mathbb U$-small, while $\Vect_\Lambda$ and the resulting functor and module categories are regarded as $\mathbb V$-categories. Terms such as ``all limits and colimits,'' ``arbitrary direct sums,'' and ``cocontinuous'' are understood relative to $\mathbb U$. Thus $\calV_C$ is a Grothendieck abelian category in the usual universe-relative sense, and its limits and colimits are computed pointwise.

If $D\subseteq C$, let
\[
\pi_{D,C}:S_D\longrightarrow S_C,
\qquad
gD\longmapsto gC.
\]
It is easy to see that  the pullback
\[
\pi^*_{D,C}:\calV_C\longrightarrow\calV_D
\]
is faithful and exact and preserves all limits and colimits. It need not be full; this is why we specify the global morphism spaces explicitly.

\subsection{The filtered pseudocolimit}

\begin{definition}\label{def:split-pseudocolimit}
 We define the uniformly smooth split category $\calV_H^{\uSm}$  as follows:  objects are  pairs $(C,V)$ with $C\in\calC_H$ and $V\in\calV_C$. For two objects,  morphisms are defined by
\[
\operatorname{Hom}_{\calV_H^{\uSm}}
\bigl((C,V),(C',V')\bigr)
=
\varinjlim_{D\subseteq C\cap C'}
\operatorname{Hom}_{\calV_D}
\bigl(\pi^*_{D,C}V,\pi^*_{D,C'}V'\bigr),
\]
where $D$ ranges through $\calC_H$. Composition is formed after passage to a common refinement.

This is an explicit model of the filtered pseudocolimit of the categories $\calV_C$ in the 2-category of $\Lambda$-linear categories.
\end{definition}

\begin{theorem}\label{thm:split-abelian}
The category $\calV_H^{\uSm}$ is $\Lambda$-linear and abelian. Finite limits and finite colimits are represented after passage to one common congruence level.
\end{theorem}

\begin{proof}
Linearity and composition follow from filteredness. By the defining equivalence relation in the filtered colimit, equality of two morphism representatives is witnessed after passage to a common refinement. Let
\[
f:(C,V)\longrightarrow(C',V')
\]
be represented at a common refinement $D$ by a morphism
\[
f_D:\pi^*_{D,C}V\longrightarrow\pi^*_{D,C'}V'.
\]
Define $\ker(f)$ and $\operatorname{coker}(f)$ to be represented at level $D$ by the pointwise kernel and cokernel of $f_D$. If $f$ is represented at a finer level, exactness identifies the new kernel and cokernel with the pullbacks of the old ones. If two representatives become equal only after a further refinement, the same argument there identifies the resulting kernels and cokernels. Thus these objects are independent of all choices.

To check the universal property of the kernel, let $g:X\to(C,V)$ satisfy $fg=0$ in the pseudocolimit. Both $g$ and the equality $fg=0$ are represented after a further refinement $E\subseteq D$. Exactness of $\pi^*_{E,D}$ identifies the pullback of $\ker(f_D)$ with the kernel of the pullback of $f_D$, so $g$ factors uniquely at level $E$. The same argument proves uniqueness in the colimit. Cokernels are analogous. Finite biproducts are formed at a common refinement.

Since all transition functors are pointwise pullbacks, images and coimages are computed fiberwise and therefore preserved under refinement. At level $D$, the canonical morphism $\operatorname{coim}(f_D)\to\operatorname{im}(f_D)$ is an isomorphism. Exact pullback preserves this isomorphism at every refinement, and hence $\operatorname{coim}(f)\to\operatorname{im}(f)$ is an isomorphism in the pseudocolimit. Thus the category is abelian.
\end{proof}

\begin{warning}
No global cocompleteness statement is made for $\calV_H^{\uSm}$. A family of objects can require unboundedly fine stabilizers, while a single object of the pseudocolimit is represented at one level. Moreover, filtered colimits of morphism spaces do not in general commute with infinite products, so even a pointwise direct sum chosen at a common starting level need not satisfy the global coproduct universal property. All arbitrary direct sums below are taken inside one fixed category $\calV_C$.
\end{warning}

For $h\in H$, right multiplication gives a bijection
\[
r_h:S_{hCh^{-1}}\longrightarrow S_C,
\qquad
g(hCh^{-1})\longmapsto ghC.
\]
Define
\[
R_h^C=(r_h)^*:\calV_C\longrightarrow\calV_{hCh^{-1}}.
\]
These functors commute with refinement and satisfy
\[
R_{h_1}^{h_2Ch_2^{-1}}\circ R_{h_2}^C=R_{h_1h_2}^C,
\qquad
R_1^C=\id.
\]

\begin{theorem}[Uniformly smooth right action]\label{thm:right-action}
The functors $R_h$ define an exact strict action of $H$ on $\calV_H^{\uSm}$. Every object and morphism admits an open stabilizer in the chosen congruence basis, acting strictly trivially on a representative level.
\end{theorem}

\begin{proof}
The strict action law holds at every level and is compatible with refinement. Each $R_h^C$ is an equivalence of split categories, hence exact, so the induced global functor is exact.

An object and a morphism are represented at a common level $C$. For $c\in C$, one has $cCc^{-1}=C$ and $gcC=gC$ for every $g\in H$, so $R_c^C$ is literally the identity on that representative. This gives a common open stabilizer.
\end{proof}

\section{Twisted Borel equivariantization and the principal series 2-representation}

\subsection{The uniformly smooth principal series 2-representation}
For $C\in\calC_H$,  we use  Proposition~\ref{prop:multiplier-action} by setting
\[
Q=B_H,\qquad S=S_C,\qquad L(q)=L_\alpha(q),\qquad\sigma=\sigma_\alpha.
\]
Thus
\[
(T_b^{\alpha,C}V)_s:=L_\alpha(b)\otimes V_{b^{-1}s}
\]
defines an exact weak 2-representation of $(B_H)_{\disc}$ on $\calV_C$. Its composition constraints is induced by $m_{b,b'}$.

\begin{definition}
The fixed-level categorical principal series is the equivariantization
\[
\calI_C(\alpha):=\calV_C^{B_H,\sigma_\alpha}.
\]
By Proposition~\ref{prop:equiv-descent}, it is canonically equivalent to
\[
\Fun_{\sigma_\alpha}(B_H\ltimes S_C,\Vect_\Lambda).
\]
\end{definition}

\begin{proposition}\label{prop:fixed-groth}
	For each $C$, the category $\calI_C(\alpha)$ is a $\Lambda$-linear Grothendieck abelian category. Limits, colimits, kernels, and cokernels are computed on the underlying fibers.
\end{proposition}

\begin{proof}
	Via Proposition~\ref{prop:equiv-descent}, this is the category of representations of a small action groupoid with a scalar central twist. The forgetful functor to $\prod_{s\in S_C}\Vect_\Lambda$ creates limits and colimits, which are therefore pointwise. The cocycle identity, equivalently the pentagon of Proposition~\ref{prop:multiplier-action}, guarantees consistency of the descent maps. Grothendieckness follows from the standard description as modules over a small $\Lambda$-linear category.
\end{proof}

For $D\subseteq C$, the $B_H$-equivariant projection $S_D\to S_C$ strictly intertwines the weak Borel actions and induces an exact pullback
\[
\pi^*_{D,C}:\calI_C(\alpha)\longrightarrow\calI_D(\alpha).
\]

\begin{definition}
The uniformly smooth principal series category is the explicit pseudocolimit
\[
\calI_F^{\uSm}(\alpha)
:=
\mathop{\mathrm{colim}}_{C\in\calC_H}\calI_C(\alpha),
\]
understood in the concrete sense of Definition~\ref{def:split-pseudocolimit}.
\end{definition}

For $h\in H$, right translation strictly intertwines the weak Borel actions:
\[
R_h^CT_b^{\alpha,C}
=
T_b^{\alpha,hCh^{-1}}R_h^C.
\]
Indeed, $r_h(b^{-1}s)=b^{-1}r_h(s)$ and the line $L_\alpha(b)$ is unchanged. Compatibility with the composition constraints is literal. Hence $R_h^C$ descends to an exact equivalence on the equivariantizations.

\begin{theorem}[Uniformly smooth principal series 2-representation]\label{thm:principal}
The category $\calI_F^{\uSm}(\alpha)$ is $\Lambda$-linear and abelian. Right translation defines an exact strict uniformly smooth 2-representation
\[
\rho_\alpha:H_{\disc}\longrightarrow
\Aut^{\ex}_\Lambda\bigl(\calI_F^{\uSm}(\alpha)\bigr),
\qquad
h\longmapsto R_h.
\]
Every object and every morphism has an open stabilizer in $\calC_H$.
\end{theorem}

\begin{proof}
The abelian statement follows from the proof of Theorem~\ref{thm:split-abelian}, since every transition functor is faithful and exact. The strict intertwining above transports equivariantization data, and the identities
\[
R_{h_1}R_{h_2}=R_{h_1h_2},
\qquad
R_1=\id
\]
hold strictly in the chosen pseudocolimit model. They are the composition constraints and unit of the monoidal functor $\rho_\alpha$. Uniform smoothness follows from Theorem~\ref{thm:right-action}: an object and a morphism represented at level $C$ are fixed by $C$.
\end{proof}

\subsection{The determinant 2-group and inflation}\label{subsec:determinant-inflation}
Fix a model, denoted $2\text{-}\operatorname{Tate}_k$, of the exact category $2\text{-}\operatorname{Tate}^{\mathrm{el}}(\Vect_k^{\fd})$ of elementary 2-Tate objects over $k$ in the sense of~\cite{braunling2016tate}. The field $F=k((u))((t))$ is an elementary 2-Tate object, and
\[
W_{F,n}=F_L^n\oplus F_R^n
\]
is an object of $2\text{-}\operatorname{Tate}_k$. We let 
\[
\Aut_{2\text{-}\operatorname{Tate}_k}(W_{F,n})
\]
denote the ordinary group of automorphisms of $W_{F,n}$ in this exact category. Multiplication by block-diagonal matrices defines a homomorphism
\[
\iota_H:H_{\disc}
\longrightarrow
\Aut_{2\text{-}\operatorname{Tate}_k}(W_{F,n})_{\disc}.
\]

We take the existence of a compatible graded determinant extension for the chosen 2-Tate model as an input and fix a coherent central extension of 2-groups
\[
\Pic_k^{\mathbb Z}
\longrightarrow
\widetilde{\Aut}^{\det}_{2\text{-}\operatorname{Tate}_k}(W_{F,n})
\longrightarrow
\Aut_{2\text{-}\operatorname{Tate}_k}(W_{F,n})_{\disc}.
\]
 Here $\Pic_k^{\mathbb Z}$ is the graded Picard groupoid of graded one-dimensional $k$-vector spaces, with $\pi_0\simeq\mathbb Z$ and $\pi_1\simeq k^\times$.   Central extensions by graded lines or Picard groupoids associated with double-loop and two-dimensional local-field constructions appear in~\cite{frenkel2012gerbal,osipov2011categorical}; the exact-category model of Tate objects and its index map are developed in~\cite{braunling2016tate,braunling2018index}, and a $K$-theoretic construction of higher Tate central extensions is given in~\cite{saito2014higher}. 

We define the determinant 2-group over $H$ to be the homotopy pullback
\[
H^{\det}_{F,n}
:=
H_{\disc}
\mathop{\times}^{h}_{\Aut_{2\text{-}\operatorname{Tate}_k}(W_{F,n})_{\disc}}
\widetilde{\Aut}^{\det}_{2\text{-}\operatorname{Tate}_k}(W_{F,n}),
\]
taken in the 2-category of coherent 2-groups. Let
\[
p:H^{\det}_{F,n}\longrightarrow H_{\disc}
\]
be the first projection. By construction, the homotopy fiber of $p$ over the unit object is equivalent to $\Pic_k^{\mathbb Z}$. Only the projection $p$ is used below; replacing the chosen determinant-extension model by an equivalent one yields an equivalent inflated action.

\begin{corollary}[Inflation along the determinant extension]\label{cor:det-lift}
Precomposition with $p$ gives an exact $p$-uniformly smooth 2-representation
\[
\rho_\alpha^{\det}:=\rho_\alpha\circ p:
H^{\det}_{F,n}\longrightarrow
\Aut^{\ex}_\Lambda\bigl(\calI_F^{\uSm}(\alpha)\bigr).
\]
The determinant Picard kernel acts coherently trivially: every kernel object is sent, through the unit coherence of the factorization, to the identity autoequivalence, and every kernel morphism is sent to the identity natural transformation.
\end{corollary}

\begin{proof}
The projection $p$ and the action $\rho_\alpha$ are monoidal functors, so their composite is monoidal. The homotopy fiber of $p$ maps to the unit object and identity morphism of $H_{\disc}$; applying $\rho_\alpha$ therefore gives the identity autoequivalence and its identity natural transformation. Exactness and $p$-uniform smoothness are inherited from the underlying $H$-action by  the definition of uniformly smooth.
\end{proof}

\subsection{The spherical object}
Let $\calC_{H_0}$ be the directed family of finite intersections of the groups $C^{\mathrm{dbl}}_{ij}$. Every $C\in\calC_{H_0}$ is normal in $H_0$.

\begin{lemma}[Integral triviality]\label{lem:integral-triviality}
For $b,b'\in B_H\cap H_0$, one has $\sigma_\alpha(b,b')=1$.
\end{lemma}

\begin{proof}
Every diagonal entry of $b$ and $b'$ is a $t$-adic unit. If $v_t(x)=v_t(y)=0$, the tame boundary $\partial_t\{x,y\}$ equals $1$, hence $\nu_F\{x,y\}=0$. Every exponent in the definition of $\sigma_\alpha$ therefore vanishes.
\end{proof}

For $C\in\calC_{H_0}$, put
\[
X_C=H_0/C,\qquad Q_0=B_H\cap H_0,
\qquad O_C=B_HH_0/C\subseteq S_C.
\]
We give a descent construction used below. The inclusion of action
groupoids
\[
\iota_C:Q_0\ltimes X_C\longrightarrow B_H\ltimes O_C
\]
is an equivalence. It is essentially surjective because every point of $O_C$
has the form $bkC$ with $b\in B_H$ and $k\in H_0$, and the arrow $b$ carries
$kC\in X_C$ to $bkC$. It is fully faithful because, if $a\in B_H$ carries
$kC$ to $k'C$ for $k,k'\in H_0$, then
\[
k'^{-1}ak\in C\subseteq H_0,
\]
so $a\in B_H\cap H_0=Q_0$. Thus the arrows between objects of $X_C$ in the
larger action groupoid are exactly the arrows already present in
$Q_0\ltimes X_C$.

In the standard trivializations of the multiplicative lines, write
$e_q\in L_\alpha(q)$ for $q\in Q_0$. Integral triviality gives
\[
m_{q,q'}(e_q\otimes e_{q'})=e_{qq'},
\qquad q,q'\in Q_0.
\]
Consequently the constant line $\underline{\Lambda}_{X_C}$ carries a
$\sigma_\alpha|_{Q_0}$-twisted equivariant structure with structure maps
\[
\theta_q(x):L_\alpha(q)\otimes\Lambda
\xrightarrow{\sim}\Lambda,
\qquad e_q\otimes\lambda\longmapsto\lambda,
\qquad x\in X_C.
\]
Indeed, the twisted descent identity
\[
\theta_{qq'}(x)\circ(m_{q,q'}\otimes\id)
=
\theta_q(q'x)\circ(\id\otimes\theta_{q'}(x))
\]
reduces to the displayed multiplicativity of the vectors $e_q$.
Here the vectors $e_q$ belong to the fixed presentation of the multiplicative
line $L_\alpha$; they are not choices of nonzero vectors in the fibers of the
object being constructed.

Transporting this twisted representation across $\iota_C$ gives a
$\sigma_\alpha$-twisted line bundle $\calL^{\mathrm{sph}}_{\alpha,C}$ on
$O_C$. Concretely, for $b\in B_H$ and $k\in H_0$, its fiber may be represented
as
\[
\bigl(\calL^{\mathrm{sph}}_{\alpha,C}\bigr)_{bkC}
\simeq L_\alpha(b)\otimes_\Lambda\Lambda.
\]
This description is independent of the presentation of the point. Namely, if
$bkC=b'k'C$, then $q=b'^{-1}b$ belongs to $Q_0$, satisfies $qkC=k'C$, and
$b=b'q$. The two descriptions of the fiber are identified by the span of
isomorphisms
\[
L_\alpha(b)\otimes\Lambda
\xleftarrow{\ m_{b',q}\otimes\id\ }
L_\alpha(b')\otimes L_\alpha(q)\otimes\Lambda
\xrightarrow{\ \id\otimes\theta_q(kC)\ }
L_\alpha(b')\otimes\Lambda.
\]
The cocycle identity for $\sigma_\alpha$, equivalently the twisted descent
identity above, makes these identifications compatible with a third
presentation. Since $O_C$ is $B_H$-stable, the line bundle can be extended by
zero to all of $S_C$, and the twisted equivariance maps extend uniquely across
the zero fibers.

This is the line-valued analogue of the usual spherical vector in the function
model of an induced representation: the value $1$ on the spherical support is
replaced by a one-dimensional line, while the scalar transformation law is
replaced by the multiplicative lines $L_\alpha(b)$ and their coherence maps.
The support and fiber dimensions do not depend on $\alpha$; the dependence on
$\alpha$ lies in the twisted descent data.

\begin{definition}\label{def:spherical-object}
	The spherical object $v_{\alpha,C}\in\calI_C(\alpha)$ is the extension by zero
	of $\calL^{\mathrm{sph}}_{\alpha,C}$ from $O_C$ to $S_C$. This definition fixes
	the line bundle and its descent data, but it does not choose a basis in each
	one-dimensional fiber.
\end{definition}

\begin{proposition}\label{prop:spherical-object}
For every $C\in\calC_{H_0}$, the object $v_{\alpha,C}$ is well defined, has one-dimensional fibers on $O_C$ and zero fibers off $O_C$, and is right $H_0$-equivariant. The objects are compatible with refinement, and any pair $(C,v_{\alpha,C})$ represents an object
\[
v_\alpha\in\calI_F^{\uSm}(\alpha)
\]
that is independent of $C$ up to the canonical refinement isomorphisms.
\end{proposition}

\begin{proof}
The descent statement and the assertions about the fibers were established above. For $k_0\in H_0$, normality of $C$ gives a right-translation bijection
\[
r_{k_0}:X_C\longrightarrow X_C,
\qquad xC\longmapsto xk_0C.
\]
It commutes strictly with the left $Q_0$-action. Hence the identity maps on the fibers of the constant line define a $Q_0$-equivariant isomorphism
\[
\varphi_{k_0}:r_{k_0}^{*}\underline{\Lambda}_{X_C}
\xrightarrow{\sim}
\underline{\Lambda}_{X_C}.
\]
Because all these maps are identities on the underlying copy of $\Lambda$, they satisfy the right-action coherence
\[
\varphi_{k_1k_2}
=
\varphi_{k_1}\circ r_{k_1}^{*}(\varphi_{k_2})
\]
literally. Transport across the equivalence $\iota_C$ and extension by zero therefore give a coherent right $H_0$-equivariant structure on $v_{\alpha,C}$.

If $D\subseteq C$ in $\calC_{H_0}$, the preimage of $O_C$ in $S_D$ is $O_D$. The constant twisted line on $X_C$ pulls back to the constant twisted line on $X_D$, and the identity right-translation maps pull back to the corresponding identity maps. Thus
\[
\pi_{D,C}^{*}v_{\alpha,C}\simeq v_{\alpha,D}
\]
compatibly with the $B_H$-descent and right $H_0$-equivariant structures. Consequently the objects $(C,v_{\alpha,C})$ are canonically isomorphic in the filtered pseudocolimit; we denote any one of them by $v_\alpha$.
\end{proof}

\section{Integration and spherical inner product }
\label{sec:fesenko-inner-product}

For $C\in\calC_H$, set
\[
N_H(C)=\{h\in H:hCh^{-1}=C\}.
\]
If $h\in N_H(C)$, right translation gives an exact endofunctor $R_h^C$ of
$\calI_C(\alpha)$.  Its fixed-level categorical trace is defined by
\[
\Tr_{\mathrm{cat}}(R_h^C)
:=\Nat\bigl(\id_{\calI_C(\alpha)},R_h^C\bigr).
\]

In this section, we give a positive Hermitian structure on a concrete spherical part of the
$2$-representation.   The required
integral on $\GL_n(F)$ is the general-rank integral constructed by Morrow from
Fesenko's integration on $F$
\cite{fesenko2003analysis,morrow2008integration,morrow2010integration}.
After the normalization adopted below, Waller's measure gives the same
rank-two volumes and bi-invariance
\cite[Theorem~5.5 and Corollaries~5.6--5.7]{waller2019measure}.
We use only a finite-step class of functions for which all integrals, changes of
variables, and positivity statements can be proved directly.

\subsection{Fesenko--Morrow integration, positivity, and unitary parameters}

\[
G=\GL_n(F),
\qquad
K_0=\GL_n(\calO_2),
\qquad
Q_G=B\cap K_0.
\]
For $a\in F^\times$, we define
\[
|a|_F
=
q^{-v_u(\overline{a t^{-v_t(a)}})}X^{v_t(a)}.
\]
This is the formal absolute value used in the lifted integration theory.
Identify $M_n(F)$ with $F^{n^2}$ by the standard matrix coordinates and use
the corresponding repeated product integral.  Morrow's space
$\mathcal L(G)$ consists of functions $\phi$ for which
$\tau\mapsto\phi(\tau)|\det\tau|_F^{-n}$ is the restriction of a function
$f_\phi$ in his Fubini class on $M_n(F)$.  Its integral is
\begin{equation}
\mathcal I_n(\phi)
=
\int_{M_n(F)}f_\phi(x)\,dx.
\label{eq:morrow-gl-n-integral}
\end{equation}
The definition is independent of the extension
\cite[Definition~5.2 and Remark~5.3]{morrow2008integration}.  Morrow proves
that this space is stable under left and right translations and that $\mathcal I_n$ is
invariant under both \cite[Proposition~5.4]{morrow2008integration}.  We extend scalars
along $\mathbb C(X)\hookrightarrow\Lambda$ whenever necessary.

Set
\[
d_{n,q}=\prod_{r=1}^n(1-q^{-r}),
\qquad
c_{n,q}=d_{n,q}^{-1},
\qquad
m_n=\frac{n(n-1)}2.
\]
For $i,j\ge1$, write
\[
J_{i,j}=u^it^j\calO_2,
\qquad
K_{i,j}=I_n+M_n(J_{i,j}),
\qquad
P_{i,j}=Q_GK_{i,j}.
\]
We order $\mathbb R((X))$ by the leading-coefficient convention: a nonzero
Laurent series is positive when the coefficient of its lowest nonzero power of
$X$ is positive.  Thus $X$ is a positive infinitesimal.

\begin{theorem}[Normalized Fesenko--Morrow volumes in arbitrary rank]
\label{thm:general-rank-fesenko-morrow}
For every $n\ge1$, the integration functional
\[
\mu_n=c_{n,q}\mathcal I_n: \mathcal L (G)  \to \mathbb C((X))
\]
is left and right translation invariant on Morrow's integrable class and is
normalized by
\[
\mu_n(K_0)=\mu_n (\mathbf 1_{K_0})=1.
\]
The characteristic functions of $K_0$, $K_{i,j}$, $P_{i,j}$, and all their
left or right translates belong to that class.  Moreover,
\begin{align}
\mu_n(K_{i,j})
&=c_{n,q}q^{-n^2i}X^{n^2j},
\label{eq:general-congruence-volume}\\
\mu_n(P_{i,j})
&=c_{n,q}(1-q^{-1})^nq^{-m_ni}X^{m_nj}.
\label{eq:general-borel-congruence-volume}
\end{align}
In particular, these values lie in the positive cone of
$\mathbb R((X))$.
\end{theorem}

\begin{proof}
Bi-invariance is Morrow's change-of-variables theorem just cited.  His lifting
comparison for $G$ identifies the integral of the lift of a residue-field
function with its Haar integral
\cite[Proposition~5.8]{morrow2008integration}.  Normalize additive Haar
measure on $E=k((u))$ by $\operatorname{vol}(k[[u]])=1$.  Then
\[
\operatorname{vol}\bigl(\GL_n(k[[u]])\bigr)
=
\frac{|\GL_n(k)|}{q^{n^2}}
=
\prod_{r=1}^n(1-q^{-r})
=d_{n,q}.
\]
The zero extension of $\mathbf1_{\GL_n(k[[u]])}$ to $M_n(E)$ is a
Schwartz--Bruhat function.  Its Morrow lift is exactly
$\mathbf1_{K_0}$,  because  we have     $\pi : \GL_n(\mathcal O_1) \to \GL_n(E), \pi^{-1}(\GL_n(k[[u]])) = \GL_n(\mathcal O_2)$.    Since
$|\det g|_F=1$  for  $ g \in K_0$,   \cite[Proposition~5.8]{morrow2008integration} gives
$\mathcal I_n(\mathbf1_{K_0})=d_{n,q}$. Thus $\mu_n(K_0)= c_{n,q}  \mathcal I_n(\mathbf1_{K_0})= d_{n,q}^{-1} d_{n,q }=1$ proving the normalization.

 By Fesenko's basic volume formula
\cite[Section~1]{fesenko2003analysis},  we have
\begin{equation}
\operatorname{vol}_F(a+J_{i,j})=q^{-i}X^j.
\label{eq:fesenko-basic-box}
\end{equation}
Every element of $K_{i,j}$ has determinant of formal absolute value one, and
$K_{i,j}$ is the affine product of $n^2$ copies of $J_{i,j}$ around $I_n$.
Its characteristic function, extended by zero to $M_n(F)$, is the lift 
centered at $I_n$ and $t$-depth $j$ in every coordinate of the Schwartz--Bruhat
function $\mathbf1_{M_n(u^ik[[u]])}$ on $M_n(E)$.  It therefore belongs to
Morrow's matrix Fubini class
\cite[Definition~4.3 and Remark~4.2]{morrow2008integration}; since the
determinant factor is one on its support, it defines an element of
$\mathcal L(G)$.  The repeated lifting formula
\cite[Remark~3.9]{morrow2008integration} gives
\[
\mathcal I_n(\mathbf1_{K_{i,j}})
=(q^{-i}X^j)^{n^2},
\]
hence  $\mu_n(K_{i,j})= c_{n,q} \mathcal I_n(\mathbf 1_{K_{i,j}}) =c_{n,q}q^{-n^2i}X^{n^2j}$, which proves \eqref{eq:general-congruence-volume}.

We next claim that
\begin{equation}
P_{i,j}
=
\left\{
(a_{rs})\in M_n(F):
\begin{array}{ll}
a_{rr}\in\calO_2^\times,&1\le r\le n,\\
a_{rs}\in\calO_2,&r<s,\\
a_{rs}\in J_{i,j},&r>s
\end{array}
\right\}.
\label{eq:borel-congruence-box}
\end{equation}
The inclusion from left to right follows from the ideal property of
$J_{i,j}$.  Conversely, for a matrix $a$ in the displayed set, let $q$ be
its upper-triangular part.  The diagonal entries of $q$ are units, so
$q\in Q_G$, while
$q^{-1}a\equiv I_n\pmod{M_n(J_{i,j})}$.  Thus
$q^{-1}a\in K_{i,j}$, proving the claim.

Again $|\det a|_F=1$ throughout this set.  Each diagonal coordinate has
volume $1-q^{-1}$, each upper off-diagonal coordinate has volume one, and
each of the $m_n$ lower coordinates has volume $q^{-i}X^j$.  Extended by zero
to $M_n(F)$, the characteristic function in
\eqref{eq:borel-congruence-box} is the mixed-coordinate lift of the product of
$\mathbf1_{k[[u]]^\times}$ on the diagonal,
$\mathbf1_{k[[u]]}$ above the diagonal, and
$\mathbf1_{u^ik[[u]]}$ below the diagonal, using $t$-depth zero on and above
the diagonal and $t$-depth $j$ below it.  This residue-field function is
Schwartz--Bruhat and hence GL-Fubini, so its lift belongs to Morrow's matrix
Fubini class \cite[Definition~4.3 and Remark~4.2]{morrow2008integration}.
The lifting formula \cite[Remark~3.9]{morrow2008integration} therefore gives
\[
\mathcal I_n(\mathbf1_{P_{i,j}})
=(1-q^{-1})^n(q^{-i}X^j)^{m_n}.
\]
This proves \eqref{eq:general-borel-congruence-volume}.  Stability of the
translated characteristic functions follows from bi-invariance of the
integrable class.
\end{proof}

\begin{remark}
	The result  of $\mathbb R((X))$-measure of $\GL_n(F)$ of a two-dimensional local field $F$  was already obtained  by Waller \cite{waller2019approach}.
\end{remark}

Return to the doubled integral subgroup
\[
H_0=K_{0,L}\times K_{0,R},
\qquad
Q_0=Q_{G,L}\times Q_{G,R}=B_H\cap H_0.
\]
On finite sums of products of the functions occurring in
Theorem~\ref{thm:general-rank-fesenko-morrow}, define
\[
\mu_0=\mu_n\boxtimes\mu_n.
\]
Thus $\mu_0(H_0)=1$, and $\mu_0$ is bi-invariant under $H_0$.  Put
\[
C_{i,j}^{\mathrm{dbl}}=K_{i,j,L}\times K_{i,j,R},
\qquad
P_{i,j}^{\mathrm{dbl}}
=Q_0C_{i,j}^{\mathrm{dbl}}
=P_{i,j,L}\times P_{i,j,R}.
\]

\begin{corollary}[Doubled general-rank volumes]
\label{cor:doubled-general-rank-volumes}
For every $i,j\ge1$,
\begin{align}
\mu_0(C_{i,j}^{\mathrm{dbl}})
&=c_{n,q}^2q^{-2n^2i}X^{2n^2j},
\label{eq:doubled-congruence-volume}\\
\mu_0(P_{i,j}^{\mathrm{dbl}})
&=c_{n,q}^2(1-q^{-1})^{2n}
  q^{-2m_ni}X^{2m_nj}.
\label{eq:doubled-borel-congruence-volume}
\end{align}
\end{corollary}

\begin{proof}
Both identities are the products of the corresponding one-factor identities
in Theorem~\ref{thm:general-rank-fesenko-morrow}.
\end{proof}

 We recall that $\calC_{H_0}$ is the family of finite
intersections of the groups $C_{i,j}^{\mathrm{dbl}}$,  that every member is normal
in $H_0$, and that
\[
X_C=H_0/C.
\]
The special form of the rank-two ideals makes this family particularly
simple.

\begin{lemma}[The congruence chain]
\label{lem:congruence-chain}
The ideals $J_{i,j}$, and hence the groups $C_{i,j}^{\mathrm{dbl}}$ and
$P_{i,j}^{\mathrm{dbl}}$, form a chain.  Explicitly,
\[
J_{i',j'}\subseteq J_{i,j}
\quad\Longleftrightarrow\quad
j'>j\ \text{or}\ (j'=j\ \text{and}\ i'\ge i).
\]
Consequently every member of $\calC_{H_0}$ is one of the standard groups
$C_{i,j}^{\mathrm{dbl}}$.  If $i'\ge i$ and the $t$-depth is fixed, then
\begin{equation}
[P_{i,j}^{\mathrm{dbl}}:P_{i',j}^{\mathrm{dbl}}]
=q^{2m_n(i'-i)}.
\label{eq:same-depth-index}
\end{equation}
\end{lemma}

\begin{proof}
Using
\[
J_{i,j}=u^it^jk[[u]]+t^{j+1}E[[t]],
\]
the asserted inclusions are immediate.  Finite intersections therefore equal
the smallest member involved.  For the index in one factor, put
$R=\calO_2/J_{i',j}$ and
$\mathfrak a=J_{i,j}/J_{i',j}$.  Then
$\mathfrak a^2=0$ and
$|\mathfrak a|=q^{i'-i}$.  The image of $P_{i,j}$ in $\GL_n(R)$ is the
inverse image of the upper Borel over $R/\mathfrak a$, whereas the image of
$P_{i',j}$ is the upper Borel over $R$.  The first group has
$|\mathfrak a|^{n^2}$ lifts over each residue matrix and the second has
$|\mathfrak a|^{n(n+1)/2}$ such lifts.  Their quotient therefore has
cardinality
$|\mathfrak a|^{n^2-n(n+1)/2}=q^{m_n(i'-i)}$.
Taking the two factors gives \eqref{eq:same-depth-index}.
\end{proof}

For later use, the standard family $\calC_{H_0}$ is cofinal in
$\calC_H$.  Indeed, for fixed $h=(h_L,h_R)\in H$, the two $F$-linear maps
$A\mapsto h_L^{-1}Ah_L$ and $A\mapsto h_R^{-1}Ah_R$ on $M_n(F)$ have
only finitely many matrix coefficients.  Choose $N\ge0$ so that multiplying
every nonzero coefficient by $t^N$ puts it in $\calO_2$.  Then
\[
C_{i,j+N}^{\mathrm{dbl}}
\subseteq hC_{i,j}^{\mathrm{dbl}}h^{-1}
\qquad(i,j\ge1).
\]
Taking a common deeper standard level for finitely many conjugates proves
cofinality.

Fix $C\in\calC_{H_0}$.  For every $E\in\calC_{H_0}$ with $C\subseteq E$,
put
\[
P_E=Q_0E.
\]
Since $E$ is normal in $H_0$, this is a subgroup.  For $g\in H_0$, the subset
$P_Eg/C\subseteq X_C$ is left $Q_0$-invariant.  Define
\begin{equation}
\mathcal A_C^{\mathrm{cyl}}
=
\operatorname{span}_{\Lambda}
\left(
\mathbf1_{X_C},
\ \mathbf1_{P_Eg/C}
\;\middle|\;
C\subseteq E,\ g\in H_0
\right)
\subseteq
\operatorname{Fun}(X_C,\Lambda)^{Q_0}.
\label{eq:cylindrical-algebra}
\end{equation}
For $f\in\mathcal A_C^{\mathrm{cyl}}$, define
\begin{equation}
\tau_C(f)
=
\int_{H_0} f(kC)\,d\mu_0(k).
\label{eq:cylindrical-trace}
\end{equation}
The pullback to $H_0$ is a finite linear combination of characteristic
functions covered by Theorem~\ref{thm:general-rank-fesenko-morrow}, so this
integral is defined without any additional hypothesis.

To formulate positivity, recall that $\Lambda=\mathbb C((X))$ and equip it
with the coefficientwise involution
\[
\left(\sum_{m\ge m_0}a_mX^m\right)^\dagger
=
\sum_{m\ge m_0}\overline{a_m}X^m,
\qquad X^\dagger=X.
\]

\begin{proposition}[Faithful positive cylindrical trace]
\label{prop:positive-cylindrical-trace}
The space $\mathcal A_C^{\mathrm{cyl}}$ is a unital commutative
$\dagger$-algebra under pointwise operations.  It is stable under right
translation by $H_0$.  The functional $\tau_C$ is unital,
$\dagger$-compatible, right-translation invariant, and faithful positive:
\begin{equation}
\tau_C(f^\dagger f)>0
\qquad
\text{for every }0\ne f\in\mathcal A_C^{\mathrm{cyl}}.
\label{eq:faithful-positivity}
\end{equation}
If $D\subseteq C$ in $\calC_{H_0}$, then pullback along
$X_D\to X_C$ maps $\mathcal A_C^{\mathrm{cyl}}$ into
$\mathcal A_D^{\mathrm{cyl}}$ and satisfies
\begin{equation}
\tau_D(\pi_{D,C}^*f)=\tau_C(f).
\label{eq:cylindrical-refinement-trace}
\end{equation}
\end{proposition}

\begin{proof}
Cosets of two subgroups in the chain of Lemma~\ref{lem:congruence-chain} are
either disjoint or one contains the other.  Hence the intersection of two
sets occurring in \eqref{eq:cylindrical-algebra} is empty or another such
coset.  This proves closure under products; closure under $\dagger$ is
coefficientwise.  Right translation sends $P_Eg/C$ to $P_Egh/C$.
Unitality follows from $\mu_0(H_0)=1$, and right invariance follows from
bi-invariance of the product Morrow integral.  Every generating coset has the
real volume displayed in Corollary~\ref{cor:doubled-general-rank-volumes};
by linearity this also gives
$\tau_C(f^\dagger)=\tau_C(f)^\dagger$.

It remains to prove strict positivity.  If $n=1$, then
$P_{i,j}^{\mathrm{dbl}}=H_0$ for every $i,j$, so
$\mathcal A_C^{\mathrm{cyl}}=\Lambda\mathbf1_{X_C}$ and the assertion is
immediate.  Assume $n\ge2$.  A finite collection of the cosets in
\eqref{eq:cylindrical-algebra} generates a finite Boolean algebra.  Its
nonempty atoms are of the form
\[
A=P_Eg\setminus\bigcup_{\ell=1}^r P_{E_\ell}g_\ell,
\]
where the subcosets in the union are pairwise disjoint and properly contained
in $P_Eg$.  If a subcoset has larger $t$-depth than $P_Eg$, its measure has
strictly larger $X$-order by
\eqref{eq:doubled-borel-congruence-volume}.  At the same $t$-depth, pass to a
common deepest $u$-level.  Formula~\eqref{eq:same-depth-index} partitions the
parent into finitely many cosets of equal measure; because $A$ is nonempty,
the removed cosets are fewer than the full index.  It follows that
$\mu_0(A)$ has positive leading coefficient.  The same argument applies to
the complementary atom in $H_0$, whose leading term is $1$.

Now write a nonzero $f$ as a finite disjoint step function
$f=\sum_A a_A\mathbf1_A$ on these atoms.  Then
\[
\tau_C(f^\dagger f)
=
\sum_A a_A^\dagger a_A\,\mu_0(A).
\]
Every nonzero summand is positive in the ordered field
$\mathbb R((X))$, and at least one occurs.  This proves
\eqref{eq:faithful-positivity}.

Finally, the inverse image of $P_Eg/C$ in $X_D$ is $P_Eg/D$: indeed $C$ and
$D$ are normal in $H_0$ and are contained in $E$.  Both sides of
\eqref{eq:cylindrical-refinement-trace} are therefore the integral over
$H_0$ of the same function.
\end{proof}

A \emph{formal pre-Hilbert $\Lambda$-space} will mean a $\Lambda$-vector space
with a positive-definite Hermitian form taking diagonal values in
$\mathbb R((X))_{\ge0}$.  No Archimedean completion is understood.  We write
$\PreHilb^X_{\Lambda}$ for the category of such spaces and adjointable linear
maps.

For a parameter
$\alpha=(\alpha_1,\ldots,\alpha_n)\in(\Lambda^\times)^n$, define the
Hermitian-dual parameter
\[
\alpha^\vee
=
\bigl((\alpha_1^\dagger)^{-1},\ldots,
      (\alpha_n^\dagger)^{-1}\bigr).
\]
We call $\alpha$ \emph{unitary} if
\begin{equation}
\alpha_r^\dagger\alpha_r=1
\qquad(1\le r\le n).
\label{eq:unitary-parameter}
\end{equation}

For a $\Lambda$-vector space $V$, write $V^{\dagger\vee}$ for its
Hermitian algebraic dual, namely the $\Lambda$-space of maps
$\ell:V\to\Lambda$ satisfying $\ell(av)=a^\dagger\ell(v)$.

\begin{proposition}[Hermitian dual of the multiplier]
\label{prop:unitary-multiplier}
For all $b,b'\in B_H$,
\[
\bigl(\sigma_\alpha(b,b')^\dagger\bigr)^{-1}
=
\sigma_{\alpha^\vee}(b,b').
\]
Consequently fiberwise Hermitian duality and duality of the multiplicative
lines define a contravariant functor from the $\alpha$-twisted to the
$\alpha^\vee$-twisted equivariant category.  On the full subcategories of
objects with finite-dimensional, equivalently dualizable, fibers, this functor
is an anti-equivalence.  If $\alpha$ is unitary, then
$\alpha^\vee=\alpha$, and the standard lines
$L_\alpha(b)=\Lambda e_b$ carry Hermitian metrics with $\|e_b\|=1$ for
which every multiplication map
\[
m_{b,b'}:L_\alpha(b)\otimes L_\alpha(b')
\longrightarrow L_\alpha(bb')
\]
is unitary.
\end{proposition}

\begin{proof}
Taking the dagger inverse of the defining formula gives
\[
\bigl(\sigma_\alpha(b,b')^\dagger\bigr)^{-1}
=
\prod_{r=1}^n
\bigl((\alpha_r^\dagger)^{-1}\bigr)^{-\nu_F\{x_r(b),y_r(b')\}}
=
\sigma_{\alpha^\vee}(b,b').
\]
Dualizing a twisted descent map and then inverting the resulting isomorphism
gives the descent map on the Hermitian dual.  The multiplication on the dual
lines has scalar $(\sigma_\alpha(b,b')^\dagger)^{-1}$, hence is the
$\alpha^\vee$-twisted multiplication by the displayed identity.  This gives
the stated contravariant functor.  On finite-dimensional fibers, the canonical
double-dual maps are isomorphisms, so the functor is an anti-equivalence.
Under \eqref{eq:unitary-parameter},
\[
\sigma_\alpha(b,b')^\dagger\sigma_\alpha(b,b')
=
\prod_{r=1}^n
(\alpha_r^\dagger\alpha_r)^{-\nu_F\{x_r(b),y_r(b')\}}
=1.
\]
Thus every multiplier value is dagger-unitary, so the displayed multiplication
maps preserve the standard metrics.
\end{proof}

\subsection{The spherical dagger category and its Hom-valued inner product}

Assume from now on that \eqref{eq:unitary-parameter} holds.  Let
\[
\calI_C(\alpha)
=
\Fun_{\sigma_\alpha}(B_H\ltimes S_C,\Vect_{\Lambda})
\]
be the fixed-level category, here we identify the two categories via the canonical equivalence 。.  Under the integral-picture equivalence
\[
Q_0\ltimes X_C\simeq B_H\ltimes O_C
\]
of Subsection~5.3, the spherical object $v_{\alpha,C}$ is the constant line on
$X_C$ with trivial $Q_0$-descent.  Hence
\[
\operatorname{Hom}_{\calI_C(\alpha)}
(v_{\alpha,C}^{\oplus r},v_{\alpha,C}^{\oplus s})
\cong
M_{s\times r}
\bigl(\operatorname{Fun}(X_C,\Lambda)^{Q_0}\bigr).
\]

\begin{definition}
	
Define an additive dagger subcategory
$\mathcal S_{C}^{\mathrm{cyl}}(\alpha)\subseteq
\calI_C(\alpha)$ as follows.  Its objects are
$v_{\alpha,C}^{\oplus r}$ for $r\ge0$, and
\begin{equation}
\operatorname{Hom}_{\mathcal S_C^{\mathrm{cyl}}(\alpha)}
\bigl(v_{\alpha,C}^{\oplus r},v_{\alpha,C}^{\oplus s}\bigr)
=
M_{s\times r}(\mathcal A_C^{\mathrm{cyl}}).
\label{eq:cylindrical-homs}
\end{equation}
Composition is pointwise matrix multiplication, and
$\Phi^\dagger(kC)=\Phi(kC)^\dagger$  denotes conjugate transpose.  
\end{definition}

The refinement
functors define a filtered pseudocolimit
\[
\mathcal S_F^{\mathrm{cyl},\uSm}(\alpha)
:=
\varinjlim_{C\in\calC_{H_0}}
\mathcal S_C^{\mathrm{cyl}}(\alpha),
\]
which, by the cofinality just proved and injectivity of pullback, maps
faithfully to
$\varinjlim_{C\in\calC_{H_0}}\calI_C(\alpha)$ and hence to the uniformly
smooth category.

\begin{definition}
For two objects $V,W \in  \mathcal S_C^{\mathrm{cyl}}(\alpha)$, we define the Hom-valued spherical inner product by
\begin{equation}
\underline{\langle V,W\rangle}^{(2)}_{C,\mu}
=
\operatorname{Hom}_{\mathcal S_C^{\mathrm{cyl}}(\alpha)}(V,W),
\label{eq:hom-valued-inner-product}
\end{equation}
equipped with the scalar Hermitian form
\begin{equation}
(\Phi,\Psi)_{C,\mu}
=
\tau_C\!\left(\operatorname{tr}(\Phi^\dagger\Psi)\right)
=
\int_{H_0}
\operatorname{tr}\!\left(\Phi(kC)^\dagger\Psi(kC)\right)
\,d\mu_0(k).
\label{eq:fesenko-hom-form}
\end{equation}
\end{definition}
For morphisms $a:V'\to V$ and $b:W\to W'$, its functorial map is
\[
\Phi\longmapsto b\Phi a:
\operatorname{Hom}(V,W)\longrightarrow\operatorname{Hom}(V',W').
\]
Thus the inner product first assigns a Hom-space to a pair of objects and then
uses the Fesenko--Morrow integral to make that Hom-space a formal
pre-Hilbert space.  This is the spherical analogue of a Hilbert enrichment for a unitary
$2$-representation; compare categorical inner products in
\cite{ganter2015inner}.

\begin{theorem}[General-rank spherical unitarization]
\label{thm:general-rank-spherical-unitarization}
For every $n\ge1$ and every unitary parameter $\alpha$, the forms
\eqref{eq:fesenko-hom-form} have the following properties.
\begin{enumerate}
\item They are positive-definite Hermitian forms.  In particular, no quotient
by a null radical is required, and
\[
\|\id_{v_{\alpha,C}}\|_{C,\mu}^2=1.
\]
\item The assignment \eqref{eq:hom-valued-inner-product} is a bifunctor
\[
\bigl(\mathcal S_C^{\mathrm{cyl}}(\alpha)\bigr)^{\op}
\times\mathcal S_C^{\mathrm{cyl}}(\alpha)
\longrightarrow\PreHilb^X_{\Lambda}.
\]
The dagger is antiunitary, and composition is adjointable.  More precisely,
whenever the compositions are defined,
\begin{align*}
(A\Phi,\Psi)_{C,\mu}
&=(\Phi,A^\dagger\Psi)_{C,\mu},\\
(\Phi B,\Psi)_{C,\mu}
&=(\Phi,\Psi B^\dagger)_{C,\mu}.
\end{align*}
\item If $D\subseteq C$, the refinement functor
\[
\pi_{D,C}^*:
\mathcal S_C^{\mathrm{cyl}}(\alpha)
\longrightarrow
\mathcal S_D^{\mathrm{cyl}}(\alpha)
\]
is an isometric dagger functor.  Consequently
$\mathcal S_F^{\mathrm{cyl},\uSm}(\alpha)$ carries a well-defined
Hom-valued formal $2$-inner product.
\item The restriction of the right-translation $2$-representation
$\rho_\alpha$ to $H_0$ preserves
$\mathcal S_F^{\mathrm{cyl},\uSm}(\alpha)$ and its inner product up to the
coherent right $H_0$-equivariant structure of $v_{\alpha,C}$ from
Proposition~\ref{prop:spherical-object}.  If
$\varphi_h:R^C_hv_{\alpha,C}\xrightarrow{\sim}v_{\alpha,C}$ denotes that
structure map, set
\[
U_h(\Phi)=\varphi_h^{\oplus s}\,R^C_h(\Phi)\,
(\varphi_h^{\oplus r})^{-1}
\qquad
(\Phi:v_{\alpha,C}^{\oplus r}\to v_{\alpha,C}^{\oplus s}).
\]
In the integral picture this is
\begin{equation*}
(U_h\Phi)(kC)=\Phi(khC),
\qquad h\in H_0,
\label{eq:k0-unitary-action}
\end{equation*}
and
\[
(U_h\Phi,U_h\Psi)_{C,\mu}=(\Phi,\Psi)_{C,\mu}.
\]
Thus the $H_0$-restriction of the spherical part of
$\rho_\alpha$ is a unitary $2$-representation in the Hom-enriched sense of
\eqref{eq:hom-valued-inner-product}.
\end{enumerate}
\end{theorem}

\begin{proof}
For $\Phi\ne0$, at least one matrix entry $f$ is nonzero.  Pointwise,
$\operatorname{tr}(\Phi^\dagger\Phi)$ is the sum of the scalar functions
$f_{ab}^\dagger f_{ab}$.  By
Proposition~\ref{prop:positive-cylindrical-trace}, every term has
nonnegative integral and the chosen nonzero entry has strictly positive
integral.  Hence
\[
(\Phi,\Phi)_{C,\mu}>0.
\]
Hermitian symmetry and sesquilinearity are immediate.  The identity of the
spherical line is the constant function one, so its squared norm is
$\mu_0(H_0)=1$.

Postcomposition and precomposition give the two functorial variables in
\eqref{eq:hom-valued-inner-product}.  The first adjointness identity follows
pointwise from $(A\Phi)^\dagger=\Phi^\dagger A^\dagger$.  The second also
uses the ordinary finite-matrix identity
$\operatorname{tr}(XY)=\operatorname{tr}(YX)$ for rectangular products.
Thus the induced linear maps are adjointable.  Moreover, cyclicity of the
finite matrix trace gives directly
\[
(\Phi^\dagger,\Psi^\dagger)_{\underline{\langle W,V\rangle}^{(2)}_{C,\mu}}
=
(\Psi,\Phi)_{\underline{\langle V,W\rangle}^{(2)}_{C,\mu}},
\]
so dagger is antiunitary.

The refinement assertion follows from
Proposition~\ref{prop:positive-cylindrical-trace} entry by entry: both norms
are the integral over $H_0$ of the same matrix coefficient function.  The
pullback is injective because $X_D\to X_C$ is surjective, so the isometric
direct limit remains positive-definite.

Every $C\in\calC_{H_0}$ is normal in $H_0$.  Hence right translation by
$h\in H_0$ stays at level $C$, preserves the cylindrical algebra, and acts as
in \eqref{eq:k0-unitary-action}.  Right invariance of $\mu_0$ gives the
last displayed isometry, and pointwise conjugate transpose gives
$U_h(\Phi^\dagger)=U_h(\Phi)^\dagger$.  The coherence maps of the spherical
$H_0$-equivariant structure are identities in the integral picture.  Hence
$U_{h_1}U_{h_2}=U_{h_1h_2}$, and these coherence maps are unitary and
compatible with the strict multiplication law of the original
right-translation action.
\end{proof}

\subsection{Explicit orthogonal vectors and the Waller specialization}

The preceding theorem contains explicit orthogonal systems.  Let
$E=C_{a,b}^{\mathrm{dbl}}$ contain the fixed level $C$.  For right cosets of
$P_E$ in $H_0$ one has
\begin{equation}
\left(
\mathbf1_{P_Eg/C},\mathbf1_{P_Eg'/C}
\right)_{C,\mu}
=
\begin{cases}
\mu_0(P_E),&P_Eg=P_Eg',\\
0,&P_Eg\ne P_Eg'.
\end{cases}
\label{eq:orthogonal-coset-vectors}
\end{equation}
By \eqref{eq:doubled-borel-congruence-volume}, the vectors
\begin{equation}
e_{E,g}
=
c_{n,q}^{-1}(1-q^{-1})^{-n}
q^{m_na}X^{-m_nb}\,
\mathbf1_{P_Eg/C}
\label{eq:normalized-spherical-cosets}
\end{equation}
have norm one.  Thus the two valuation directions appear separately: the
$u$-depth contributes a power of $q$, while the $t$-depth contributes a power
of the infinitesimal $X$.

There is also an ambient, not necessarily $Q_0$-invariant, normalized
congruence vector in the finite-step function space on $H_0$:
\begin{equation}
e_{i,j}^{\mathrm{dbl}}
=
c_{n,q}^{-1}q^{n^2i}X^{-n^2j}
\mathbf1_{C_{i,j}^{\mathrm{dbl}}}.
\label{eq:normalized-general-congruence}
\end{equation}
Its norm is one by \eqref{eq:doubled-congruence-volume}.  The distinction is
important: \eqref{eq:normalized-spherical-cosets} is an endomorphism of a
spherical object, whereas \eqref{eq:normalized-general-congruence} belongs to
the ambient $H_0$-function space.

For $n=2$,
\[
c_{2,q}
=
\frac{q^3}{(q^2-1)(q-1)}.
\]
Consequently \eqref{eq:general-congruence-volume} becomes
\[
\mu_2(K_{i,j})
=
\frac{q^3}{(q^2-1)(q-1)}q^{-4i}X^{4j},
\]
which is exactly Waller's normalization.  Waller also proves directly in rank
two that the resulting measure is both left and right translation invariant
and that his coordinate integral agrees, up to the normalization constant,
with Morrow's integral \cite[Appendix~A]{waller2019measure}.  Thus the
rank-two formula is not a separate assumption but the first nontrivial
instance of Theorem~\ref{thm:general-rank-fesenko-morrow}.

\bibliographystyle{alpha}
\bibliography{myreferences}
\end{document}